\documentclass[a4paper,oneside,11pt]{amsart}

\usepackage{amssymb,amsfonts,amsmath,amsxtra,mathtools,
mathrsfs,placeins,graphicx,verbatim,stmaryrd,hyperref,cite,color,amsthm}
\usepackage[all]{xy}
\usepackage[T1]{fontenc}
\usepackage{lmodern}

\xyoption{line}

\usepackage{cite}
\usepackage{accents}
\usepackage{amsmath,stackengine}

\makeatletter
\@namedef{subjclassname@2020}{
	\textup{2020} Mathematics Subject Classification}
\makeatother

\newtheorem{theorem}{Theorem}[section]
\newtheorem{cor}[theorem]{Corollary}
\newtheorem{lem}[theorem]{Lemma}
\newtheorem{prop}[theorem]{Proposition}

\theoremstyle{definition}

\newtheorem{defi}[theorem]{Definition}
\newtheorem{rem}[theorem]{Remark}

\numberwithin{equation}{section}

\DeclareMathOperator{\End}{End}

\DeclareMathOperator{\Hom}{Hom}
\DeclareMathOperator{\uHom}{\underline{Hom}}
\DeclareMathOperator{\oHom}{\overline{Hom}}

\DeclareMathOperator{\Map}{Map}
\DeclareMathOperator{\Ndg}{N_{dg}}
\DeclareMathOperator{\Ncoh}{N_{coh}}

\DeclareMathOperator{\Ob}{Ob}

\newcommand{\MCd}{\mathsf {MC_{dg}}}
\newcommand{\MC}{\mathsf{MC}}

\newcommand{\Ba}{\mathrm{B}}

\newcommand{\dgCat}{\mathsf{dgCat}} 
\newcommand{\dgCatM}{\mathsf{dgCat}'_{\mathrm{Mor}}}

\newcommand{\ptdco}{{\mathsf{ptdCoa}^*}}
\newcommand{\ptdcoM}{{\mathsf{ptdCoa}^*_{\mathrm{Mor}}}}
\newcommand{\ptdcono}{{\mathsf{ptdCoa}}}
\newcommand{\strptdco}{{\mathsf{ptdCoa}^{\mathrm{str}}}}
\newcommand{\grqui}{{\mathsf{grQuiv}}}

\newcommand{\sSet}{\mathsf{sSet}}
\newcommand{\qCat}{\mathsf{qCat}}
\newcommand{\sCat}{\mathsf{sCat}}

\newcommand{\Ch}{\mathsf{dgVect}}
\newcommand{\op}{^{\operatorname{op}}}
\newcommand{\MCp}[2]{\overline{\mathsf {MC}}\{#1, #2\}}

\newcommand{\HHH}{C^*_{\mathsf{HH}}}
\newcommand{\HH}{\mathsf{HH}}

\newcommand{\utilde}[1]{\widetilde{#1}}

\def\ground{\mathbf{k}}

\def\id{\operatorname{id}}

\calclayout

\theoremstyle{theorem}
\newtheorem{innercustomthm}{Theorem} 
\newenvironment{customthm}[1]
{\renewcommand\theinnercustomthm{#1}\innercustomthm}
{\endinnercustomthm}

\newtheorem{innercustomcor}{Corollary} 
\newenvironment{customcor}[1]
{\renewcommand\theinnercustomcor{#1}\innercustomcor}
{\endinnercustomcor}

\begin{document}
	\bibliographystyle{../hsiam2}
	\begin{abstract}
	It is well-known that the category of small dg categories $\dgCat$, though it is monoidal, does not form a monoidal model category. In this paper we construct
	a monoidal model structure on the category of pointed curved coalgebras $\ptdco$ and show that the Quillen equivalence relating it to $\dgCat$ is monoidal.  
We also show that $\dgCat$ is a $\ptdco$-enriched model category.
As a consequence, the homotopy category of $\dgCat$ is closed monoidal and is equivalent as a closed monoidal category to the homotopy category of $\ptdco$. 
	In particular, this gives a conceptual construction of a derived internal hom in $\dgCat$. 
	As an application we obtain a new description of simplicial mapping spaces in $\dgCat$ and a calculation of their homotopy groups in terms of Hochschild cohomology groups, reproducing and slightly generalizing well-known results of To\"en. Comparing our approach to To\"en's, we also obtain a description of the core of Lurie's dg nerve in terms of the ordinary nerve of a discrete category.
	\end{abstract}
	\title[Enriched Koszul duality for dg categories]{Enriched Koszul duality for dg categories}
	\author{J. Holstein}
	\address{Department of Mathematics\\
		Universit\"at Hamburg\\
		20146 Hamburg\\
		Germany	
	}
	\email{julian.holstein@uni-hamburg.de}
	
	\author{A.~Lazarev}
	\thanks{This work was partially supported by EPSRC grant EP/T029455/1}	
	\address{Department of Mathematics and Statistics\\
		Lancaster University\\
		Lancaster LA1 4YF\\United Kingdom}
	\email{a.lazarev@lancaster.ac.uk}
	\maketitle
	\tableofcontents

\section{Introduction}
The category $\dgCat$ of small differential graded (dg) categories has an internal homotopy theory, underpinned by a Quillen model structure constructed by Tabuada \cite{Tabu05} whose weak equivalences are quasi-equivalences. 
There is also a closed monoidal structure on $\dgCat$, however since the tensor product of two dg categories is not a Quillen bifunctor, e.g. the tensor product of two cofibrant dg categories is not itself cofibrant, the two structures are not compatible.
Thus, the internal hom in $\dgCat$ does not determine an internal hom in the homotopy category $Ho(\dgCat)$, even though the tensor product does lift to $Ho(\dgCat)$. Nevertheless, To\"en showed in \cite{Toen06} that $Ho(\dgCat)$ does have a closed monoidal structure, cf. also \cite{Stel15} for an alternative approach. The resulting derived internal hom  is constructed using bimodules of special kind (`quasi-representable functors'). Borrowing the terminology of algebraic geometry, this result shows that all functors between dg-categories are of Fourier-Mukai type, cf. \cite{Stel12} regarding this point of view.
For the Morita model structure on $\dgCat$ Tabuada gave a construction of internal homs by considering a category of localizing pairs \cite{Tabuada10a}.

Looking at this issue from a different angle, Kontsevich suggested that the category of unital $A_\infty$-functors  between two dg categories can be taken as the derived internal hom between them. A complete proof of this statement was obtained only recently, in \cite{Canonaco19}.

The purpose of this paper is to provide another, more structured version of the derived hom for dg categories. Our starting point is categorical Koszul duality developed in \cite{Holstein6} and which states, roughly, that the category of (small) dg categories is Quillen equivalent to the category of coalgebras of a special kind. At the same time it is known that the category of coalgebras (unlike that of algebras) does possess an internal hom making it a closed monoidal category (this point of view and its various ramifications are explained in \cite{Anel13}). One can hope, therefore, that the internal hom in coalgebras is homotopically better behaved than that in dg categories, and admits a straightforward lift to homotopy categories. Provided that Koszul duality is compatible with monoidal structures on coalgebras, this would give a derived internal hom for dg categories. 

The programme thus outlined is carried out in the present paper.
More precisely, it establishes a closed monoidal structure on the category $\ptdco$ of pointed curved coalgebras and shows that this structure is compatible with the model structure on $\ptdco$ and the Koszul adjunction to dg categories. Kontsevich's characterization of the derived internal hom in $\dgCat$ is an immediate consequence of these results. 

Note that there is a close analogy between dg-categories and $\infty$-categories: just like in dg-categories, simplicial categories do not possess a well-behaved internal hom, and this problem is resolved by replacing simplicial categories by simplicial sets with the Joyal model structure which does possess a monoidal model structure and thus, a well-behaved internal hom. In fact, this is more than an analogy; there exists a direct relationship between the differential graded and simplicial pictures, it is explained in Remark \ref{rem:diagram} below.

Recall that a coalgebra $C$ is pointed if its coradical is a direct sum of copies of the ground field $\ground$ that we fix throughout the paper. 
A pointed curved coalgebra is a curved coalgebra that is pointed and has a splitting of the coradical satisfying some compatibilities which we recall below, see Definition \ref{def:semicoalgebraworking}.
We denote by $\ptdco$ the category of pointed curved coalgebras equipped with a final object.

According to \cite{Holstein6} there is a natural model structure on $\ptdco$ and a Quillen equivalence $\Omega: \ptdco \rightleftarrows \dgCat': \Ba$ between  pointed curved coalgebras and dg categories (with the Dwyer-Kan model structure).  Here $\dgCat'$ is an equivalent model for $\dgCat$ that will be defined below. $\Omega$ and $\Ba$ are suitable versions of the cobar and bar construction.

In this setting we prove the following results:

\begin{customthm}{\ref{thm:coalgmonoidal}}
	There is a closed monoidal model structure on $\ptdco$.
\end{customthm}

Moreover the cobar construction is quasi-strong monoidal in the sense that it induces a strong monoidal functor on homotopy categories (Lemma \ref{lem:omegamonoidal}), and it induces an equivalence of monoidal category $Ho(\ptdco) \cong Ho(\dgCat)$ (Corollary \ref{cor:monoidalequivalence}).

The compatibility between $\ptdco$ and $\dgCat$ goes beyond the homotopy categories. 
While $\dgCat$ is not a monoidal model category, it is enriched, tensored and cotensored (powered) over $\ptdco$ and this structure is compatible with the model structures of $\dgCat'$ and $\ptdco$.
We say for short $\dgCat$ is a $\ptdco$-enriched model category.
\begin{customthm}{\ref{thm:cotensor}}
	The category $\dgCat'$ is a $\ptdco$-enriched model category.
\end{customthm}

Our results have immediate consequences for describing internal homs in $Ho(\dgCat)$.
\begin{customcor}{\ref{cor:dgcathom}}
	The internal hom in $Ho(\dgCat)$ may be computed as (a small modification of) the Maurer-Cartan category of a convolution category $R\uHom(D, D') \simeq \MCp {BD} {D'}$.
\end{customcor}
Unravelling definitions, this internal hom is identified with the category of unital $A_\infty$-functors as proposed by Kontsevich.

We also obtain a new description of mapping spaces in $\dgCat$, that differs from the classical description in \cite{Toen06}.

\begin{customthm}{\ref{thm:dgmapping}}
	Given two dg categories $D, D'$ the mapping space $\Map(D, D')$ is weakly equivalent to the core of $\Ndg R\uHom(D, D')$ where $\Ndg$ denotes Lurie's dg nerve.
\end{customthm}

Combining this with Toen's characterization of the mapping space, we obtain a new description of the core of the dg nerve $\Ndg$ in terms of the \emph{ordinary} nerve of the category of weak equivalences (Corollary \ref{cor:nervecore} below).

By defining a Morita model structure on $\ptdco$ we may apply our techniques also to the Morita model structure on $\dgCat$ and compute internal homs similarly.

As an application, we compute the homotopy groups of simplicial mapping spaces between dg categories in terms of Hochschild cohomology, reproducing and slightly generalizing the well-known result of To\"en \cite{Toen06}.

\begin{customthm}{\ref{thm:hhcomputation}}
	Let $G: D \to D'$ be a functor of dg categories. 
	We then have $\HH^0(D,D')^\times \cong \pi_1(\Map(D,D'), G)$ and 
	$\HH^{i}(D, D') \cong \pi_{1-i} (\Map(D,D'), G)$ for $i < 0$.
\end{customthm}

Finally, we mention that analogous results hold  for ordinary dg Koszul duality (i.e.\ for the Quillen equivalence between augmented dg algebras and conilpotent dg coalgebras). In this context, the monoidal structure on coalgebras is given by a smash-product (as opposed to a tensor product) and it lacks a monoidal unit, which brings about specific subtleties. This theory is being developed in \cite{bjor}.

\subsection{Outline}
We first recall some concepts and definitions as well as the main result of \cite{Holstein6} in the remainder of this section.
Then Section \ref{sec:resolutions} provides some auxiliary results on resolving dg categories by free dg categories and pointed curved coalgebras by cofree ones.
In Section \ref{sec:convolution} we construct the convolution category structure on homs from a pointed curved coalgebra to a dg category, see Definition \ref{defi:gradedconvolutionmonoid}. This is the main ingredient in constructing the MC category $\MCp C D$ (Definition \ref{defi:mcp}) and the closed monoidal structure on pointed curved coalgebras in Section \ref{sec:closedcoalgebras}.
The compatibility with the model structure on $\ptdco$ is provided in Section \ref{sec:monoidalmodel}.
Section \ref{sec:coalgebrasandcategories} then compares the monoidal structures on $\ptdco$ and $\dgCat'$ and shows that dg categories form a $\ptdco$-enriched model category. 
We show the same is true for the Morita model structure on $\dgCat'$ if we adjust the model structure on $\ptdco$.
We conclude with applications to Hochschild cohomology in Section \ref{sec:Hochschild}.

\subsection{Definitions, notation and conventions}
We refer to \cite{Holstein6} for a more detailed overview of the background material described below. The symbol  $\ground$ stands for a fixed ground field; $\ground$-linear Hom sets and Hom complexes will be denoted by $\Hom$ while internal hom objects will be denoted by $\uHom$. The modifier `$\ground$-linear' will usually be omitted later on as no other ground rings or fields will be considered (e.g. an `algebra' will stand for a `$\ground$-algebra' etc.) Mapping spaces in model categories will be denoted as $\Map(-,-)$.

Let $\dgCat$ be the category of \emph{differential graded (dg) categories}, considered as a model category with its Dwyer-Kan(-Bergner-Tabuada) model structure where weak equivalences are given by quasi-equivalences \cite{Tabu05}. (When considering the Morita model structure on $\dgCat$ at the end of Section \ref{sec:coalgebrasandcategories} we will make this explicit.)

We will mainly consider $\dgCat'$, the full subcategory of $\dgCat$ consisting of dg categories without zero objects together with the dg category $\bold 0$ with one object and a single zero morphism. The reason for this minor modification is that $\dgCat'$ fits better with categorical Koszul duality of \cite{Holstein6} than $\dgCat$.
The category $\dgCat'$ inherits the Dwyer-Kan model structure and the associated $\infty$-categories of $\dgCat$ and $\dgCat'$ are equivalent.

There is a natural right Quillen functor from $\dgCat$ to $\infty$-categories. We write $\qCat$ for simplicial sets with the Joyal model structure, then the \emph{dg nerve} $\Ndg: \dgCat \to \qCat$ is constructed \cite[Section 1.3.1]{Lurie11}; it was further analyzed in \cite[Section 4]{Rivera19} and \cite[Section 4]{Holstein6}. It may be restricted to $\dgCat'$.

The monoidal structure on $\dgCat$ may easily be extended to $\dgCat'$ with one adjustment: We need to define $\bold 0 \otimes D$ to be $\bold 0$ for all dg categories $D$ (the standard definition would be a dg category with the same objects as $D$ but only zero morphisms and this is not an object of $\dgCat'$).
The incluson $\dgCat' \to \dgCat$ is lax monoidal and induces a strong monoidal functor on homotopy categories. We call such functors \emph{quasi-strong monoidal}.

A \emph{curved category} is a graded category equipped with a degree 1 map $d$ on all morphism spaces satisfying the Leibniz rule, and a degree 2 curvature endomorphism $h_X$ for each object $X$ such that for $f: X \to Y$ we have $d^2f = h_Y f - f h_X$.

Let $D$ be a dg category with the set of objects $\Ob D$ and consider the coalgebra $D_0:=\ground[\Ob D]$ spanned by grouplike elements, one for every object in $D$. Then $D$ can be viewed it as a monoid in bicomodules over the cosemisimple coalgebra $D_0$. This is an example of a \emph{semialgebra}.

In particular the space of morphisms between two objects $d, d'$ is given by the cotensor product $\ground_d \boxempty_{D_0} D \boxempty_{D_0} \ground_{d'}$ where $\ground_d$ is the 1-dimensional comodule whose coaction map is induced by the inclusion of $d$ into $\Ob D$.

Let $C$ be a coalgebra. We denote by $C_0$ is its coradical, i.e.\ its maximal cosemisimple subcoalgebra and set $\overline{C}:= C/C_0$. 
We say $C$ is a \emph{pointed coalgebra} if $C_0$ is a direct sum of copies of the ground field. Such coalgebras are also called  cocomplete augmented cocategories.
We then denote the set of grouplike elements by $\Ob C$ so that $C_0 \cong \ground[\Ob C]$.

A pointed coalgebra $C$ is \emph{split} if it is equipped with a section $\epsilon: C\to C_0$ of the inclusion $C_0 \to C$.

We will now consider curved coalgebras.

As in \cite{Holstein6} many statements become easier when moving from the category of coalgebras to the opposite category of \emph{pseudocompact algebras} by taking continuous duals. Recall that a pseudocompact algebra is a topological algebra that is the projective limit of (discrete) finite-dimensional algebras.

Recall the following definition, analogous to \cite[Section 3.1]{Positselski11}.

\begin{defi}\label{def:curvedalg}
	A \emph{curved  pseudocompact algebra} $A=(A,d,h)$ is a graded  pseudocompact  algebra supplied with a derivation $d:A\to A$ (a differential) of degree 1 and an element $h\in A^2$ called the \emph{curvature} of $A$, such that $d^2(x)=[h,x]$ and $d(h)=0$ for any $x\in A$.
	
	A \emph{curved morphism} between two curved pseudocompact  algebras $A\to B$ is a pair $(f,b)$ where $f:A\to B$ is a map of graded  algebras of degree zero and $b\in B^1$ so that:
	\begin{enumerate}
		\item $f(d_Ax)=d_Bf(x)+[b,f(x)]$;
		\item $f(h_A)=h_B+d_B(b)+b^2$.
	\end{enumerate}
	Two such morphisms $(f,b)$ and $(g,c)$ are composed as $(g,c) \circ (f,b) =(g\circ f, c+g(b))$.
	In particular, every map
	$(f,b)$ can be decomposed $(f,b)=(\id,b)\circ(f,0)$.
\end{defi}

\begin{defi}\label{def:curvedcoa}
	A \emph{curved coalgebra} is a coalgebra $C$ equipped with an odd coderivation $d$ and a homogeneous linear function $h: C \to \ground$ of degree 2, called the \emph{curvature}, such that its dual $(C^*, d^*, h^*)$ is a curved pseudocompact algebra.
	
	A morphism of curved dg coalgebras from $(C, d_{C}, h_{C})$ to $(D, d_{D}, h_{D})$ is given by the data 
	$(f,a)$ where $f: C \to D$ is a morphism of graded coalgebras and $a: C \to \ground$ is a linear map of degree 1,
	such that $(f^*, a^*)$ is a curved morphisms $D^* \to C^*$.
	The composition rule is $(g,b)\circ(f,a) = (g\circ f, b \circ f + a)$.
\end{defi}

\begin{defi}\label{def:semicoalgebraworking}
	A \emph{pointed curved coalgebra} is a tuple $ (C, \Delta_{\ground}, \epsilon_{\ground}, d, h_\ground, \epsilon_C)$ such that
	\begin{itemize}
		\item $(C, \Delta_{\ground}, \epsilon_{\ground}, d, h_\ground)$ is a curved coalgebra (over $\ground$) 
		\item the restriction of $d$ to the coradical $C_0 \hookrightarrow C$ is zero,
		\item $\epsilon_C: C \to C_0$ is a a coalgebra map compatible with the differential $d$, which is left inverse to $i: C_0 \hookrightarrow C$.
	\end{itemize}
	We will often write simply $C$ or $(C, \epsilon_C)$ for $ (C, \Delta_{\ground}, \epsilon_{\ground}, d, h_\ground, \epsilon_C)$ when it does not cause confusion. 
\end{defi}
The map $\epsilon$ together with the comultiplication induce the structure of a $C_0$-bicomodule on $C$.
It follows from coassociativity that $\Delta_{\ground}$ factors as $C \xrightarrow{\Delta} C \boxempty_{C_0} C \to C \otimes C$.
Thus, $\Delta:C\to C \boxempty_{C_0} C$ and $\epsilon$ exhibit $C$ as a comonoid in $C_0$-bicomodules.
The inclusion $C_0\hookrightarrow C$ provides a coaugmentation of this comonoid.
The differential $d$ is compatible with the $C_0$-bicomodule structure and the comonoid structure given by $\Delta$ and $\epsilon$.
Note also that there is automatically a curvature $h$ with values in $C_0$, obtained by factorizing the curvature $h_{\ground}: C \to \ground$ as $\epsilon_\ground \circ (h_{\ground} \otimes \id_{C_0}) \circ \rho_{C}: C \to C \otimes C_0 \to \ground \otimes C_0 \to \ground$, where $\rho_{C}$ is the right coaction.
We define $h$ as $(h_{\ground} \otimes \id_{C_0}) \circ \rho_{C}$.

We note that in particular the zero vector space $0$ with coradical $0$ and all defining maps equal to the zero map is a pointed curved coalgebra which we will also denote by $0$.

\begin{defi}\label{def:curvedmorphism}
	A morphism $(f, a): (C, \epsilon) \to (D, \delta)$ of pointed curved coalgebras consists of
	\begin{itemize}
		\item a morphism $(f,a_\ground)$ of curved coalgebras 
		\item a factorization of $a_\ground$ as the composition $C \xrightarrow{a} D_0 \to \ground$
	\end{itemize}
	such that
	\begin{itemize}
		\item $\delta \circ f = f \circ \epsilon$,
		\item $f$ and $a$ are $D_0$-bicomodule maps,
	\end{itemize}
	where the $D_0$-bicomodule structure on $C$ is induced by the map $f: C_0\to D_0$ on coradicals induced by $f$ (this is a slight abuse of notation, but it will be clear from context what the domain of $f$ is).
	
	The composition is then defined as  
	\[
	(g, b)\circ(f,a) = (g\circ f,  b \circ f + g \circ a).
	\]
	
	The category of pointed curved coalgebras will be denoted by $\ptdcono$. 
	The same category together with a final object $*$ will be denoted by $\ptdco$.
\end{defi}

Given a pair of grouplike elements $x,y$ in a $C \in \ptdco$ 
we write $C(x,y)$ for the complex $\ground_x \boxempty_{C_0} C \boxempty_{C_0} \ground_y$, analogous to the hom space in a category.

\subsection{Recalling categorical Koszul duality}

In \cite{Holstein6} we proved that $\ptdco$ and $\dgCat'$ are Quillen equivalent using the bar and cobar constructions.

Given $C \in \ptdco$ the cobar construction $\Omega C$ is defined to have underlying graded category given by the tensor monoid $T_{C_0} \bar C[-1]$ in $C_0$-bicomodules and differential induced by differential, comultiplication and curvature.

Similarly, given a  dg category $D$ viewed as a monoid in $D_0$-bicomodules the bar construction is defined as the tensor coalgebra $T_{D_0} \bar D[1]$ where $D \cong D_0 \oplus \bar D$ is a choice of splitting of $D$ as a bicomodule. The differential and curvature are induced by differential and composition in $D$. 
Note that the splitting is not canonical and this is a source of substantial  technical difficulties (also present in the one-object situation), for the details we refer to Section 3 of \cite{Holstein6}.

The adjunction is naturally written as
\[\Hom(C, \Ba D) \cong \MC(\{\overline C, D\}) \cong \Hom(\Omega C, D)\]
where the middle term is the set of Maurer-Cartan (MC) elements in the reduced convolution category $\{\overline C, D\}$ which will be discussed in detail below, see Definitions \ref{defi:gradedconvolutionmonoid} and \ref{defi:mcp}.

We note the following special cases. Let $\bold 0$ denote the dg category with one object and only the zero morphism and $\bold \emptyset$ the empty dg category. Recall furthermore  the inital object $0$ and the final object $*$ in $\ptdco$.
Then $\Ba \bold \emptyset = 0$ and $\Omega 0 = \bold \emptyset$ by unravelling the definitions.
Furthermore we define $\Ba \bold 0 = *$ and $\Omega * = \bold 0$.

\section{Quivers, bicomodules and resolutions}\label{sec:resolutions}

We recall the category of \emph{graded $\ground$-quivers} following Keller \cite{Keller05}.
A graded $\ground$-quiver $V$ consists of a set of vertices or \emph{objects} $\Ob V$ and a graded $\ground$-module $V(x,y)$ of \emph{arrows} for every pair of objects $x,y$.

A morphism $V \to W$ consists of a map on objects $f: \Ob V \to \Ob W$ and for each pair $x,y \in \Ob V$ a morphism of arrows $V(x,y) \to W(fx, fy)$.

The category of graded  quivers will be denoted by $\grqui$.

One sees that the data of a graded quiver is equivalent to a pair $(N, C)$ where $C$ is a coalgebra of the form $\oplus \ground$ and $N$ is a graded bicomodule $N$ over it. 
One defines the correspondence by $N = \oplus_{x,y} V(x,y)$ and $C = \ground[\Ob V]$.

\begin{lem}\label{lem:closedquivers}
	There is a closed monoidal structure on $\grqui$, with $\Ob(V \otimes W) = \Ob(V) \times \Ob(W)$ and $(V \otimes W)(x,y) = V(x) \otimes V(y)$.
The internal hom is given by $\Ob \uHom(V,W) = \Hom(\Ob V, \Ob W)$ and $\Hom(f,g) = \oplus_{x,y \in \Ob V} \Hom(V(x,y), W(fx,gy))$.

\end{lem}
\begin{proof}
	Immediate from unravelling definitions.
\end{proof}

We also consider the category of \emph{augmented graded $\ground$-quivers} $\grqui^{\mathrm{aug}}$ whose objects are graded quivers $V$ together with a factorization $\ground[\Ob V] \xrightarrow{\eta} V \xrightarrow{\epsilon} \ground[\Ob V]$ of the identity on $\ground[\Ob V]$, which is defined as the quiver with $V(x,y) = \ground \delta_{x,y}$.
Morphisms of augmented quivers are morphisms of quivers compatible with the augmentation maps $\eta$ and $\epsilon$.

\begin{rem}
The tensor product defined as above of two augmented quivers is naturally augmented and $\grqui^{\mathrm{aug}}$ is a closed monoidal category. The internal hom is somewhat delicate, it is described in detail in Section Section 5.1 of \cite{Keller05}.
\end{rem}

We now show that aribtrary objects in $\ptdco$ can be resolved by cofree coalgebras, and in fact by bar constructions of dg categories. This will be important later.

Let $\mathsf{ptdgrCo}$ denote the category of \emph{pointed graded coalgebras}, i.e.\ graded coalgebras $C$ with coradical $C_0$ of the form $\oplus \ground$ that are coaugmented comonoids over $C_0$.
\begin{lem}\label{lem:cofree}
	There is a forgetful-cofree adjunction
$U: \mathsf{ptdgrCo} \rightleftarrows \grqui^{\mathrm{aug}}: G$, where the cofree functor sends a quiver $Q$ to  the tensor coalgebra over $Q_0 = \ground [\Ob Q]$.
\end{lem}
\begin{proof}
	For any fixed pair of objects $C, Q$ we have the cofree-forgetful adjunction of coalgebras in $Q_0$-bicomodules, thus we find
	$\Hom(UC, Q) = \oplus_{f: C_0 \to Q_0} \Hom_{Q_0}(f_*UC, Q)$ but as $f$ commutes with $U$ this is $$\oplus_{f: C_0 \to Q_0} \Hom_{Q_0}(Uf_*C, Q) \cong \oplus_{f: C_0 \to Q_0} \Hom_{Q_0}(f_*C, T_{Q_0}Q) \cong \Hom(C, GQ)$$ and this bijection is natural.
\end{proof}

Let $\strptdco$ denote the category of strict pointed curved coalgebras, with the same objects as $\ptdco$ but morphisms only the strict, i.e. uncurved, morphisms $f: (C, d, h) \to (C', d', h')$ given by $f: C \to C'$ compatible with differential and curvature.

\begin{lem}\label{lem:freecurvature} 
	There is a comonadic adjunction $V: \strptdco \rightleftarrows \mathsf{ptdgrCo}: H$ with left adjoint $V$ given by the functor forgetting differential and curvature. 
\end{lem}
\begin{proof}
	For simplicity we formulate the proof for pseudocompact algebras, which is contravariantly equivalent by dualizing.
	Thus we claim that there is a monadic adjunction $H^*: \mathsf{ptdgrpcAlg} \rightleftarrows  \mathsf{ptdcupcAlg^{\textrm{str}}}: V^*$ with the right adjoint $V^*$ forgetting differential and curvature.
	
	Given a pointed graded pseudocompact algebra $A$ we define $H^*A$ to be freely generated by an element $h$ in degree $2$ as well as differentials $da$ in degree $|a|+1$ for every $a \in A$ satisfying the rules for derivations and compatibility with the bimodule structure. Then $d^2(a) \coloneqq [h,a]$ and this is easily seen to be a left adjoint of $V^*$.
	
	It is clear that $V^*$ reflects isomorphisms, thus the adjunction is monadic by the Barr-Beck theorem if we can show $V^*$ preserves certain coequalizers. 
	In fact, $V^*$ preserves all coequalizers as coequalizers in the strict curved category are given by coequalizers in the graded category equipped with the induced differential and curvature. 
	
	Dualizing we see that $V \dashv H$ is comonadic.
\end{proof}

\begin{lem}\label{lem:coalgebracofree}
	Any pointed curved coalgebra $C$ is the equalizer of a diagram of cofree pointed curved coalgebras, i.e.\ pointed curved coalgebras of the form $HGV$ for a graded augmented quiver $V$. 
\end{lem}
\begin{proof}
	We consider $C$ as an object in $\strptdco$. From the composition of the comonadic adjunctions $\strptdco \rightleftarrows \mathsf{ptdgrCo}$ from Lemma \ref{lem:freecurvature} and $ \mathsf{ptdgrCo} \rightleftarrows \grqui$ from Lemma \ref{lem:cofree} we obtain a comonad $K = HGUV$ on $\strptdco$ and $C = \operatorname{eq}(KC \rightrightarrows KKC)$.
	
	Note that a composition of comonadic functors is not necessarily comonadic, but it is if the first functor ($V$ in our case) satisfies the crude monadicity theorem (which is true as $V$ preserves all equalizers), see Theorem 3.5.1 in \cite{Barr00}.
	We conclude by noting that this diagram is also an equalizer diagram in $\ptdco$, which follows from the construction of equalizers in curved coalgebras (or equivalently coequalizers in curved pseudocompact algebras) in the proof of Lemma 3.30 in \cite{Holstein6}.
\end{proof}
\begin{lem}\label{lem:cofreebar}
	Any cofree coalgebra arises as the bar construction of a dg category.
\end{lem}
	\begin{proof}
		Let $V$ be a graded augmented quiver and write $V_0 = \ground[\Ob V]$ for the quiver given by a copy of the ground field at every object and $\bar V = V/V_0$ where $\Ob \bar V = \Ob V$ and the quotient is taken for each pair of objcts.
		
		Then we define a differential graded quiver $V' = V_0 \oplus \bar V[-1] \oplus \bar V \oplus V_0[1]$ with differential given by the identity on $V_0[1]$ and $\bar V$ respectively.
		We note that there is a map of quivers $V_0 \to V'$, but the natural projection is not compatible with the differential, so this is not an augmentation.
		
		Then define the dg category $D$ by setting all composition zero except that each $V_0(x,x) \cong \ground$ is the unit at $x \in \Ob V'$.
		
		It now follows from the definitions that $\Ba D \cong HGV$.
	\end{proof}

\begin{cor}\label{cor:coalgebraresolution}
	Every object of $\ptdco$ is of the form $\operatorname{eq}(\Ba D_1 \rightrightarrows \Ba D_2)$
	for suitable dg categories $D_1$ and $D_2$.
\end{cor}
\begin{proof}
	For a pointed curved coalgebra  this is immediate from combining Lemma \ref{lem:coalgebracofree} and Lemma \ref{lem:cofreebar}. The final object $*$ is by definition the bar construction of the dg category with one object and only the zero morphism.
\end{proof}

A similar result holds for dg categories:
\begin{lem}\label{lem:freecategories}
	Any small dg category $D$ is of the form $\operatorname{coeq}(\Omega C_1 \rightrightarrows \Omega C_2)$ for pointed curved coalgebras $C_1, C_2$.
\end{lem}
\begin{proof}
	We again construct a composition of monadic adjunctions.
	Forgetting the differential provides an adjunction $U: \dgCat \rightleftarrows \mathsf{grCat}: D$, which satisfies the crude monadicity theorem, as $U$ preserves all coequalizers. (Coequalizers of differential graded categories can be computed on the underlying graded category and then equipped with a suitable differential.)
	
	There is also a monadic adjunction from graded categories to graded reflexive quivers (quivers $V$ equipped with a unit $\ground[\Ob V] \to V$), see \cite{Wolff74}.
	
	Together these give a monad $T$ on $\dgCat$ and thus any dg category is the coequalizer of free dg categories.
	
	It remains to observe that any free dg category on a reflexive quiver $V$ is the cobar construction of a coalgebra $C$ obtained by equipping $V$ with the zero comultiplication on $V/V_0$. 
\end{proof}

\section{Convolution}\label{sec:convolution}

We now introduce a convolution structure on the maps between a pointed curved coalgebra and a dg category.
We will need the construction for not necessarily counital coalgebras, so let us define a \emph{noncounital pointed graded coalgebra} to be a pair $(C_0, C)$ where $C_0$ is a coalgebra of the form $\oplus \ground$ and $C$ is a graded $C_0$ bicomodule which has a coassociative comultiplication $C \to C \boxempty_{C_0} C$, which is not necessarily counital.

\begin{defi}
	A \emph{noncounital pointed curved coalgebra} is a noncounital pointed graded coalgebra that moreover has a differential $d$ and a curvature $h: C\to C_0$ such that the square of the differential is given by the coaction of the curvature.
\end{defi}

In particular any pointed curved coalgebra $C$ is noncounital pointed curved coalgebra by forgetting the counit, and the quotient $C/C_0$ is also a noncounital pointed curved coalgebra.

Given a dg category $D$ and a possibly noncounital pointed curved coalgebra $C$ we will now construct a convolution category $\{C, D\}$ as follows.
Note that unless $C$ is counital this is a non-unital category, i.e.\ a category without units.

The set of objects of $\{C, D\}$ is given by $\Hom(C_0, D_0)  \cong \Hom_{\mathbf{Set}}(\Ob C, \Ob D)$.
Given two objects $f, g$ in $\Hom(D_0, C_0)$ we make $C$ into a $D_0$-bicomodule $(f, g)_*C$ via the composition 
\[C \to C_0 \otimes C \otimes C_0 \xrightarrow{f \otimes \id_C \otimes g} D_0 \otimes C \otimes D_0.\]
Note that for $(f,f)_* C$ we will write $f_*C$.

Then we define
\[\Hom_{\{C, D\}}(f, g) = \Hom_{D_0}((f, g)_*{C}, D).\]
We note this agrees with the internal hom of quivers from Lemma \ref{lem:closedquivers}.

	To define the composition note that for fixed $f,h$ the comultiplication $\Delta: C \to C \boxempty_{C_0} C$ induces a map $(f,h)_*C \to (f,g)_*C \boxempty_{D_0} (g,h)_* C$ of $D_0$-bicomodules for each $g$, put differently $(f,h)_*(C \boxempty_{C_0} C) \subset (f,g)_*C \boxempty_{D_0} (g,h)_*C$ (by definition of the cotensor product).

The composition $\Hom_{\{C, D\}}(f, g) \otimes \Hom_{\{C, D\}}(g, h) \to \Hom_{\{C, D\}}(f, h)$, i.e.\
$ \Hom_{D_0}((f, g)_*C, D) \otimes \Hom_{D_0}((g, h)_*C, D) \to \Hom_{D_0}((f, h)_*C, D)$ is then defined by convolution:

\[\phi \circ \psi: C \xrightarrow{\Delta} C \boxempty_{C_0} C \hookrightarrow C \boxempty_{D_0} C \xrightarrow{\phi \otimes \psi} D \boxempty_{D_0} D \xrightarrow{\mu} D
\]
or, spelling out the $D_0$-bicomodule structures:
\[(f,h)_*C \xrightarrow{\Delta} (f,h)_*(C \boxempty_{C_0} C) \hookrightarrow (f,g)_*C \boxempty_{D_0} (g,h)_* C \xrightarrow{\phi \otimes \psi} D \boxempty_{D_0} D \xrightarrow{\mu} D
\]

The composition is associative as $\Delta$ and $\mu$ are (co)associative.
Thus $\{C,D\}$ is a (nonunital) monoid in $\ground[\Hom(C_0, D_0)]$-bicomodules. 

The differential on $C$ and $D$ induces a natural differential on $\{C, D\}$.
Let $C$ have curvature $h_C: C \to C_0$ and let $\eta: D_0 \to D$ be the unit map of $D$.
Then at an object $f \mapsto \eta \circ f \circ h_C: \Hom(C_0, D_0) \to \Hom(C, D)$ defines a curvature on the convolution category, compatible with the differential. 

\begin{defi}\label{defi:gradedconvolutionmonoid}
	Given a dg category $D$ and a possibly noncounital pointed curved coalgebra $C$ the \emph{convolution category}  $\{C, D\}$ is the possibly non-unital curved category of maps from $C$ to $D$, with objects $\Hom(C_0, D_0)$,  morphisms $\Hom_{\{C, D\}}(f, g) = \Hom_{D_0}((f, g)_*{C}, D)$ and composition given by the convolution product as defined above.
\end{defi}

If $C$ has a counit $\epsilon: C \to C_0$ and $D$ has a unit $D_0 \to D$ then every object $f$ has a unit $\eta \circ f \circ \epsilon$ and $\{C, D\}$ is a (unital) curved category.

We note that we may also define a curved convolution monoid if $D$ is a curved category as long as $C$ is counital, then the curvature will aquire an additional summand $f \mapsto h_D \circ f \circ \epsilon_C$.

\begin{lem}\label{lem:convolutionmonoid}
Let $C, C'$ 
be	 (possibly noncounital) pointed graded coalgebras and $D$ a graded category.
	Then there is an isomorphism of (possibly nonunital) graded convolution categories $\{C, \{C', D\}\} \cong \{C \otimes C', D\}$.
\end{lem}
\begin{proof}
	The statement is true on the underlying bicomodules by Lemma \ref{lem:closedquivers}. 
	Next we compare the convolution products, which essentially follows by naturality of the coproduct: 
	If $f^\#, g^\#: C \to \Hom(C', D)$ are adjoint to $f, g: C \otimes C' \to D$ we apply the two compositions to $c \otimes c'$.
	We obtain $f^\#(c^{(1)})(c'^{(1)}).g^\#(c^{(2)})(c'^{(2)})$ and   $f(c^{(1)} \otimes c'^{(1)}).g(c^{(2)} \otimes c'^{(2)})$ in Sweedler notation, which of course agree.
\end{proof}

The curved version of the lemma needs an extra assumption as $\{C', D\}$ will be curved in general.

\begin{cor}\label{lem:convolutionmonoid}
	Let $C, C'$ 
	be	 (possibly noncounital) pointed curved coalgebras and $D$ a curved category.
	Then there is an isomorphism of (possibly nonunital) curved convolution categories $\{C, \{C', D\}\} \cong \{C \otimes C', D\}$ in either of the following three cases:
	\begin{itemize}
		\item  $C$ is counital, $D$  has no curvature,
		 \item $C'$ and $D$ have no curvature,
		\item $C$ and $C'$ are both counital.
	\end{itemize}
\end{cor}
\begin{proof}
The conditions of the corollary are exactly those needed to define the curvature on both sides.
It is then straightforward to check that the induced differential and curvature are  compatible with the adjunction and agree on both sides.
\end{proof}

\section{Closed monoidal structure on coalgebras}\label{sec:closedcoalgebras}

In order to define a closed monoidal structure on coalgebras we need to introduce the Maurer-Cartan (MC) category of a curved category.

\begin{defi}
	Given a curved category $D$ we construct its \emph{dg MC category} $\MCd(D)$ whose objects are pairs $(X, \xi)$ with $X$ an object of $D$ and $\xi$ an MC element in $\End(X)$, i.e\ $\xi$ satisfies $d\xi + \xi^2 + h_X = 0$ where $h_X$ is the curvature of $D$ at $X$.
	
The hom spaces are given by the complex of morphisms between twisted elements: 
\[
\Hom_{MC}((X, \xi), (X', \xi')) = \Hom_{D}(X, X')^{[\xi, \xi']} \coloneqq (\Hom_D(X, X'), d^{[\xi, \xi']})
\]
 where $d^{[\xi, \xi']}(f) = df + \xi'f -(-1)^{|f|} f \xi$.
\end{defi}

We also write $\End_D(X)^\xi$ for $\Hom_D(X)^{[\xi, \xi]}$.

Note that even if $D$ is curved $\MCd(D)$ is naturally a dg category.
We denote by $\MC(D)$ the set of objects of $\MCd(D)$.

\begin{defi}\label{defi:mcp}
	Given $C \in \ptdco$ and $D \in \dgCat$ we define the Maurer-Cartan category $\MCp C D$ to be the full subcategory of $\MCd(\{C, D\})$ whose objects are the objects of the nonunital dg category $\MCd(\{\overline{C}, D\})$. 
\end{defi}

Thus, an object of $\MCp C D$ consists of a map $f: \Ob C \to \Ob D$ together with a Maurer-Cartan element in $\xi \in \End_{\MCd(\{\overline{C}, D\})}(f) = \Hom_{D_0}(f_* \overline{C}, D)$.

\begin{rem}
Given a pointed curved coalgebra $C$ and a dg category $D$ definition \ref{defi:mcp} categorifies the set of MC elements in maps from $\overline{C}$ to $D$ which mediates the Koszul adjunction $\Omega \dashv \Ba$.
This leads naturally to consider the reduced convolution category $\{\overline {C}, D\}$, however as this is a non-unital category we need to consider morphisms coming from the larger (unital) category $\{C, D\}$. 
\end{rem}

The category of pointed curved coalgebras admits a symmetric monoidal structure given by the ordinary tensor product. 
Indeed, if $C, C'$ are objects in $\ptdco$ with curvature functions  $h:C\to C_0$ and $h':C'\to C'_0$ and counits $\epsilon:C\to C_0$ and $\epsilon':C'\to C'_0$ then $C\otimes  C'$ is likewise pointed curved with the coradical $\ground[\Ob C]\otimes\ground[\Ob C']$ and curvature $h\otimes \epsilon'+\epsilon\otimes h': C \otimes C' \to C_0 \otimes C_0'$.
	Furthermore we define $C \otimes * = * \otimes C \coloneqq *$.
	
We are now ready for our first main result.
\begin{theorem}\label{thm:ptdco_closed}
		The tensor product defines a closed monoidal structure on
	 $\ptdco$ where the internal hom $\uHom(C, C')$ is defined as $\uHom(C, \Ba D) = \Ba \MCp C D$ whenever $C' = \Ba D$.
	 We  also define $\uHom(C, *) = *$ and $\uHom(*, C) = 0$ if $C \neq *$.
\end{theorem}

Note that from this definition $\uHom(C, 0) = \Ba \MCp C \emptyset = \Ba \emptyset = 0$ and $\uHom(0, \Ba D) = \Ba \MCp 0 D = \Ba \bold 0 = *$. From the latter it follows (see the proof of Theorem \ref{thm:coalgmonoidal} below) that $\uHom(0, C) = *$ for all $C$.

\begin{rem}
It is easy to see that the tensor product makes  $\ptdco$ into a monoidal category and one may check that $\otimes$ commutes with all colimits.
As $\ptdco$ is locally presentable, the existence of an internal hom as a right adjoint to the bifunctor $\otimes$ then follows. 
The explicit description that we give is more complicated but turns out to be very fruitful.
\end{rem}
We prove some lemmas before turning to the proof of Theorem \ref{thm:ptdco_closed}.
\begin{lem}\label{lem:mcproduct} 
		
Given $C, C' \in \ptdco$ and $D \in \dgCat$ there 
is a decomposition 
\[\MC(\{\overline{C \otimes C'}, D\}) \cong 
\coprod_{f\in \Hom(C_0 \otimes C'_0, D)} 
\coprod_{\phi \in \MC(\End_{\{C_0 \otimes \overline{C'}, D\} }
	(f))} 
\MC(\End_{\{\overline {C} \otimes C', D\}}(f)^\phi)
\]
where $(f,\phi)$ is an object in $\MC(\{C_0 \otimes \overline{C'}, D\})$
and the twist by $\phi$ is induced by the natural action of $\{C_0 \otimes\overline{ C'}, D\}$ on $\{\overline {C} \otimes C', D\}$.
\end{lem}
\begin{proof}
	The decomposition $\overline{C \otimes C'} \cong C_0 \otimes \overline{C'} \oplus \overline{C} \otimes C'$ induces a bicomodule decomposition of $\{\overline{C \otimes C'}, D\}$ and the product decomposes as $(\phi,\psi).(\phi',\psi') = (\phi.\phi', \phi.\psi' + \psi.\phi' + \psi.\psi')$, i.e. $\{C_0 \otimes \overline{C'}, D\}$ is a subsemialgebra and $\{\overline{C} \otimes C', D\}$ is an ideal of the convolution category.
	In particular the second summand has an action by the first.

Thus an MC element in $\{\overline{C \otimes C'}, D\}$ is given by an object $f \in\Hom(C_0 \otimes C_0', D_0)$ together with MC elements $\phi \in \End_{\{C_0 \otimes \overline{C'}, D\} }(f) $  and $\psi \in \End_{ \{\overline{C} \otimes C', D\} }(f)^\phi$.
\end{proof}
The following lemma is the heart of constructing the closed monoidal structure. The proof looks quite complicated, but it consists mostly in unravelling notations and book-keeping in order to reduce the result to the usual tensor hom adjunction.
\begin{lem}\label{lem:hombar}
	Given $C, C' \in \ptdco$ and $D \in \dgCat$ there is a natural isomorphism in $\ptdco$: 
	\begin{equation}\label{adjunction}\Hom(C \otimes C', \Ba D) \cong \Hom(C, \uHom(C', \Ba D))
		\end{equation}
\end{lem}
\begin{proof}
	To make the argument clearer we first consider the case where $C, C'$ and $D$ have one object. In this case we may discard the object $f: \Ob C \times \Ob C' \to \Ob D$ in our considerations.
	
By the Koszul adjunction the left hand side of (\ref{adjunction}) is the set of MC elements in the convolution algebra $\{\overline{C \otimes C'}, D\}$.
There is a decomposition of vector spaces $\{\overline{C \otimes C'}, D\} \cong \{\overline C \otimes C', D\} \times \{\overline {\ground \otimes C'}, D\}$ with $ \{\overline {\ground \otimes C'}, D\}$ a subalgebra. 
From Lemma \ref{lem:mcproduct} we obtain
\begin{eqnarray*}
	\MC(\{\overline{C \otimes C'}, D\}) 
	&\cong& \coprod_{\phi \in \MC(\{\ground \otimes \overline{C'}, D\})} \MC(\{\overline {C} \otimes C', D\}^\phi) \\
	&\cong &  \coprod_{\phi \in \MC(\{\ground \otimes \overline{C'}, D\})} 
	\MC(\{\overline {C}, \{C', D\}^{\phi'}\})
\end{eqnarray*}
where the second line follows from Lemma \ref{lem:convolutionmonoid}. 
To be precise, Lemma \ref{lem:convolutionmonoid} only identifies the underlying graded algebras. 
It remains to compare the differential and curvature on both sides. 
Here $\phi'$ is defined as the image of $\phi$ in $\{C', D\}$ under the map induced by $C' \to \overline{C'}$. Then $\phi'$ is MC and thus the twisted convolution algebra $\{C', D\}^{\phi'}$ is a dg algebra and $\{\overline {C}, \{C', D\}^{\phi'}\}$ is a well-defined curved algebra. 
Unravelling definitions matches up differentials and curvatures in the two different convolution algebras and thus the MC elements agree.

Thus we write an element of the left hand side of (\ref{adjunction}) as a pair $(\phi, \psi)$ with $\phi \in \MC(\{\ground \otimes \overline {C'}, D\})$ and  $\psi \in \MC(\{\overline C, \{C', D\}^{\phi'})$.

On the right hand side of (\ref{adjunction}) we have MC elements in $\{\overline{C}, \MCp {C'} {D}\}$, i.e.\ pairs $(X, \eta)$ where $X \in \Ob\{\overline{C}, \MCp {C'} {D})\}$ and $\eta \in \End(X)$ is MC.
Thus $X$ is an MC element of $\{ \overline{C'}, D\}$, equivalent to $\phi$ on the LHS.
The MC element $\eta$ lives in $\End(X) \cong \Hom(\overline{C}, \uHom(C', D)^{X'})$, writing $X'$ for the image of $X$. Thus after identifying $X$ and $\phi$ we may identify $\eta$ with $\psi$.

We consider arbitrary objects next. 
Again we decompose the data of an element of the left hand side of (\ref{adjunction}) first.
An object of $\uHom(\overline{C \otimes C'}, BD)$ is an MC element of $\{\overline{C \otimes C'}, D\}$, so it consists of an object $f$, given by a map $f: \Ob C \times \Ob C' \to \Ob D$, and a MC element in the endomorphism algebra of $f$, i.e.\ $\Phi \in \Hom_{D_0}(f_*\overline{C \otimes C'}, D)$ satisfying $d\Phi + \Phi^2 + h = 0$ where $h$ is the curvature induced by the curvatures of $C$ and $C'$.

As above we use Lemma \ref{lem:mcproduct} to identify $\Phi$ with a pair
\[
\phi \in \MC(\Hom_{D_0}(f_* (C_0 \otimes \overline{C'}), D), \quad
\psi \in \MC(\Hom_{D_0}(f_*({\overline C} \otimes C'), D)^\phi)
		\]
		
and consider $\phi$ and $\psi$ separately. Here we spelled out the endomorphisms of $f$ as bicomodule maps, $\End_{\{C_0 \otimes \overline{C'} ,D\}}(f) \coloneqq \Hom_{D_0}(f_* (C_0 \otimes \overline{C'}), D)$.

Using the definition of comodule maps we find:
\begin{eqnarray*}
	\phi & \in & \Hom_{D_0}(f_* (C_0 \otimes \overline{C'}), D) \\
&\cong& \bigoplus_{d_1, d_2 \in \Ob D} \bigoplus_{\substack{c_1, c_1' \\ f(c_1, c'_1) = d_1}} \bigoplus_{\substack{c_2, c'_2\\ f(c_2, c_2') = d_2}} \Hom(C_0(c_1, c_2) \otimes \overline{C'}(c_1', c_2'), D(d_1, d_2))
\\
&\cong& \bigoplus_{c \in \Ob C} \bigoplus_{d_1, d_2 \in \Ob D} \bigoplus_{\substack{c_1', c_2' \in \Ob C'\\ f(c, c'_1) = d_1, f(c, c_2') = d_2}} \Hom(\overline{C'}(c_1', c_2'), D(d_1, d_2))
\end{eqnarray*}
where we used $C_0(c_1, c_2) = \ground.\delta_{c_1, c_2}$.
The MC condition says that each $\phi(c)$ is an MC element in $\Hom_{D_0}(f^{\sharp}(c)_* \overline{C'}, D)$.
We also have
\begin{eqnarray*}
	\psi &\in& \Hom_{D_0}(f_*({\overline C} \otimes C'), D)^\phi \\
	&=& 
 \bigoplus_{d_1, d_2 \in \Ob D} \bigoplus_ {\substack{c_1, c_2 \in \Ob C, \ c_1', c_2' \in \Ob C' \\ f(c_1, c'_1) = d_1, \ f(c_2, c_2') = d_2}} 
  \Hom(\overline{C}(c_1, c_2) \otimes C'(c_1', c_2'), D(d_1, d_2))^\phi \\
  	&=& 
  \bigoplus_{d_1, d_2 \in \Ob D} \bigoplus_ {\substack{c_1, c_2 \in \Ob C, \ c_1', c_2' \in \Ob C' \\ f(c_1, c'_1) = d_1, \ f(c_2, c_2') = d_2}} 
  \Hom\left(\overline{C}(c_1, c_2),  \Hom( C'(c_1', c_2'), D(d_1, d_2))^{\phi(c_1), \phi(c_2)}\right)
\end{eqnarray*}
where we used Lemmas \ref{lem:convolutionmonoid} and \ref{lem:mcproduct} again to identify the last two lines as graded algebras and then observe that differential and curvature also match up.

On the right hand side of (\ref{adjunction}) we unravel similarly. A MC element of $\uHom(C, \Ba \MCp {C'} D))$ is given by an MC element of the convolution category $\{C,  \MCp {C'} D)\}$, i.e.\ an object $\Xi: \Ob C \to \Ob(\MCp {C'} D)$ with an MC element $\eta$ in $\End(\Xi) = \Hom_{M_0}(\Xi_* C, \MCp {C'} D)$ where $M_0$ is the coalgebra $\ground[\MCp {C'} D]$.

Thus $\Xi$ sends any $c \in \Ob C$ to a pair $(f^\sharp(c), \xi(c))$ with $f^\sharp(c): \Ob C' \to \Ob D$ an object and $\xi(c) $ an MC element 
\[\xi(c) \in \End_{\MCp {C'} D}(f^\sharp) \cong \Hom_{D_0}({f^\sharp}_* \overline{C'}, D)
\]
Here we write $f^\sharp: \Ob C \to \Hom(\Ob C', \Ob D)$ as it may be seen as the adjoint of $f: \Ob C \times \Ob C' \to \Ob D$ from the LHS.

Summing over $c$ we find that $\xi$ is equivalent to a MC element in 
\begin{eqnarray*}
\bigoplus_{c \in \Ob C} \End_{\{\overline{C'}, D\}}(f^\sharp(c)) 
&\cong& \bigoplus_{c \in \Ob C}\Hom_{D_0}((f^\sharp)_* \overline{C'}, D)\\
&\cong & 
\bigoplus_{c \in \Ob C} \bigoplus_{d_1, d_2 \in \Ob D} \bigoplus_{\substack{c_1', c_2' \in \Ob C \\ f^\sharp(c)(c'_i) = d_i'}} \Hom(\overline{C'}(c'_1, c_2'), D(d_1, d_2))
\end{eqnarray*}
which is exactly the same curved algebra that $\phi$ lives in.

For the final step we need to match up $\eta$ with $\psi$.
By definition $\eta$ is a MC element in 
\begin{eqnarray*}
	\Hom_{D_0}(\Xi_* \overline {C}, \MCp {C'} D) 
	& \cong & \bigoplus_{\zeta_1, \zeta_2 \in \Ob \MCp {C'} D} 
	\bigoplus_{\substack{c_1, c_2 \in \Ob C \\ \Xi(c_i) = \zeta_i}} 
	\Hom(\overline C(c_1, c_2), \MCp{C'} D(\zeta_1, \zeta_2)) \\
	& \cong &  
	\bigoplus_{c_1, c_2 \in \Ob C} 
	\Hom( \overline C(c_1, c_2), \Hom_{\{C', D\} } (f^\sharp(c_1), f^\sharp(c_2))^{\xi(c_1), \xi(c_2)}) \\
	& \cong & 
		\bigoplus_{c_1, c_2 \in \Ob C} 
	\Hom( \overline C(c_1, c_2), \Hom_{D_0 } ((f^\sharp(c_1), f^\sharp(c_2))_* {C'}, D)^{\xi(c_1), \xi(c_2)}) \\
		& \cong & 
	\bigoplus_{c_1, c_2 \in \Ob C}  \bigoplus_{\substack{d_1, d_2 \in \Ob D\\ c_1', c_2' \in \Ob C' \\ f^\sharp(c_i)(c_i') = d_i}} 
	\Hom\left( \overline C(c_1, c_2), \Hom ({C'}(c_1', c_2'), D(d_1, d_2))^{\xi(c_1), \xi(c_2)}\right) \\
\end{eqnarray*}
where we  
used that the hom spaces in $\MCp{C} D$ are defined in terms of the nonreduced coalgebra $C'$.

Matching $\phi$ with $\xi$ this shows that $\psi$ and $\eta$ are MC elements in isomorphic curved algebras and 
the construction of the bijection is complete.

It remains to consider the special case that $C$ is $*$.
Then the left hand side is $\uHom(*, \Ba D)$ which is $0$ unless $D =\emptyset$ in which case it equals $*$.
The right hand side is $\uHom(*, \uHom(C', \Ba D))$, which is $0$ unless $\uHom(C', \Ba D)$ is $*$, which by definition is only possible if $D = \emptyset$.
A similar argument applies if $C' = *$. As all objects here are initial or final naturality is immediate.
\end{proof}

\begin{proof}[Proof of Theorem \ref{thm:ptdco_closed}]
We observe first that the construction in Lemma \ref{lem:hombar} is functorial on the full subcategory whose objects are bar constructions of dg categories. This is immediate for the first variable, for the second variable it follows from the Yoneda embedding: it suffices to construct a map $\Hom(C, \uHom(C', BD)) \to \Hom(C, \uHom(C', BD'))$ for any map $\Ba D \to BD'$.
	But such a map is equivalent to $\Hom(C \otimes C', BD) \to \Hom(C \otimes C', BD')$, induced by functoriality of the ordinary hom.
	
Thus the theorem follows from Lemma \ref{lem:hombar} as we are able to rewrite an arbitrary pointed curved coalgebra in terms of bar constructions.
Indeed, using Corollary \ref{cor:coalgebraresolution} we write $C' = \lim_i{\Ba D_i}$ and define
\[
\uHom(C, C') \coloneqq \lim_i \uHom(C, \Ba D_i) \cong \lim_i(\Ba\MCp C {D_i}).
\]	
This satisfies the adjointness isomorphism for fixed $C'$ and thus is functorial in maps $\lim_i \Ba D_i \to \lim_i \Ba D'_i$ by the Yoneda lemma again. Naturality of adjointness follows.
\end{proof}

\begin{rem}
	Note that the compatibility with limits also holds for $\uHom(*, -)$ as $\uHom(*, \lim C_i) = \lim \uHom(*, C_i)$ is $0$ unless all $C_i$ equal *, in which case we get $\uHom(*, *)=*$.
\end{rem}
\begin{cor}
	For $C, C', C'' \in \ptdco$ we have an isomorphism of pointed curved coalgebras $\uHom(C, \Hom(C', C'')) \cong \uHom(C \otimes C', C'')$.
\end{cor}
\begin{proof}
	This follows from Theorem \ref{thm:ptdco_closed} using the Yoneda Lemma.
\end{proof}

\begin{rem}
	Our construction may be compared to the internal hom of (cocomplete augmented) cocategories in terms of the internal hom of augmented quivers that is considered by Keller, see Theorem 5.3  in \cite{Keller05} and the following discussion.
	Note that \cite{Keller05} considers uncurved cocategories and uses them to express functors of augmented $A_\infty$ categories. 
	As usual in curved Koszul duality the introduction of curved coalgebras allows us to remove the augmentation.
\end{rem}

\section{A monoidal model category}\label{sec:monoidalmodel}
We next show that $\ptdco$ is in fact a monoidal model category with its model structure defined in \cite{Holstein6}.

We will first compare the monoidal structure on $\ptdco$ and $\dgCat'$. 

Let $C$ and $C'$ be pointed curved coalgebras. Let $\xi_C\in \MCp{C} {\Omega C}$ and $\xi_{C'}\in\MCp {C'} {\Omega{C'}}$ be the canonical MC elements corresponding to the identity maps $\Omega C\to\Omega C$ and $\Omega C'\to\Omega C'$.
We take their images in the curved categories $\{C, \Omega C\}$ and $\{C', \Omega C'\}$ induced by $C \to \overline{C}$ and $C' \to \overline{C'}$.
Then the element $\xi_C\otimes 1+1\otimes\xi_{C'}$ is an MC element in the curved category $\{C,\Omega C\}\otimes\{C', \Omega C'\}$. Consider the natural map of curved semialgebras
\[
m:\{C,\Omega C\}\otimes\{C', \Omega C'\}
	\to \{{C\otimes C'},\Omega C\otimes\Omega C'\}
\to \{(\overline{C \otimes C'}, \Omega C \otimes \Omega C'\}
\]
where the second map is induced by the inclusion of the kernel of the counit.
Then $m(\xi_C\otimes 1+1\otimes\xi_{C'})$ is an MC element in $\Hom(\overline{C\otimes C'},\Omega C\otimes\Omega C')$ and thus by Koszul duality, it determines a map of  dg categories
\[M:\Omega(C\otimes C')\to\Omega(C)\otimes\Omega(C').\]

\begin{lem}\label{lem:omegamonoidal}
The map $M$ defined above is a quasi-equivalence of dg categories. Thus the cobar construction is quasi-strong monoidal.
\end{lem}

\begin{proof}
	Viewing all categories as semialgebras it suffices to prove that $M$ induces a quasi-isomorphism.
	  
	Assume first that the pointed curved coalgebras $C,C'$ have no curvature. 
	Consider the cosimplicial complexes  $T_{C_0} C =\{C^{\boxempty_{C_0} i}\}_{i=0}^{\infty}$ and 
	$T_{C'_0} C' =\{C'^{\boxempty_{C'_0} i}\}_{i=0}^{\infty}$.
	The cosimplicial structures comes from considering the standard cosimplicial resolution $\{C^{\boxempty_{C_0} i}\}_{i=2}^{\infty}$ and cotensoring on the left and on the right with $C_0$.
	 Thus the coface map come from the comultiplication on $C$ (respectively $C'$)and the comodule coaction $C_0 \to C$.
	 The codegeneracy maps are induced by the counit. 
	 
	 Then $\Omega C$ and $\Omega C'$ are the totalizations of the normalized cochain complexes of these cosimplicial complexes : $\Omega(C)\cong N(T_{C_0}C)$, 
 	$\Omega(C')\cong N(T_{C_0'}C')$ 
 and $\Omega (C\otimes C')\cong N({T_{C_0 \otimes C'_0}C\otimes C'})$. 
 Next we observe that $T_{C_0 \otimes C'_0}C\otimes C' \cong T_{C_0}C \otimes T_{C'_0} C'$. We then claim that we can identify
 the map $M:\Omega(C\otimes C')\to\Omega(C)\otimes\Omega(C')$ with 
  the dual Eilenberg-Zilber map.
  As such it induces a quasi-isomorphism on total complexes, proving the lemma in this special case.
  
  To prove this claim we recall the dual Eilenberg-Zilber map $EZ^*: N(A \otimes B) \to N(A) \otimes N(B)$ that sends $a \otimes b \in A_n \otimes B_n$ to 
  \[\sum_{p+q = n} \sum_{(\lambda, \mu) \in Sh(p,q)} \epsilon(\lambda, \mu)\lambda^*(a) \otimes \mu^*(b)\]
  where we sum over all shuffles and $\epsilon(\lambda, \mu)$ denotes the sign.
  In the case at hand $A = T_{C_0}C$ and $B = T_{C_0'}C'$ and we can check that $EZ^*$ is an algebra map.

  Indeed, unravelling definitions we just have to match up shuffles on $n$ objects with pairs of shuffles on $k$ and on $n-k$ objects (with the correct signs), but this is a classical computation, see \S 17 in \cite{Eilenberg66}.
  
  Since $M$ is an algebra map by construction and on generators $C \otimes C'$ the maps $EZ^*$ and $M$ are easily seen to agree (both send $c \otimes e$ to $c \otimes 1 + 1 \otimes e$) we have shown the claim.
	
We turn to the general case where $C, C'$ may have curvature. We note that the coradical filtration on any pointed curved coalgebra has the property that its associated graded has no curvature. 
	Thus, consider the coradical filtrations on $C$, $C'$ and $C\otimes C'$ and the induced filtrations on $\Omega C\otimes\Omega C'$ and $\Omega(C\otimes C')$. 
	The map $M:\Omega(C\otimes C')\to\Omega(C)\otimes\Omega(C')$ is compatible with these filtrations, and thus it induces a map on associated graded. These are nothing but the cobar constructions of the associated graded coalgebras.
	By the special case proved above, this induced map is a quasi-isomorphism and so, $F$ is a quasi-isomorphism to begin with. 
	
	As usual we need to consider the special case $* \in \ptdco$, but both $*$ and $\Omega(*) = \bold 0$ are absorbing for the respective tensor product.
\end{proof}
\begin{rem}
	The argument above using the Eilenber-Zilber map (albeit in the dual setting of bar-constructions of algebras over a field) goes back to Cartan and Eilenberg, \cite[Chapter XI, 6]{CE56}
\end{rem}
\begin{cor}\label{cor:tensorwe}
	The tensor product with an arbitary pointed coalgebra preserves all weak equivalences.
	\end{cor}
\begin{proof}
	Let $E, C, C'$ be pointed curved coalgebras and $C \to C'$ a weak equivalence. Then we claim $E \otimes C \to E \otimes C'$ is also a weak equivalence. By definition we need to check $\Omega(E \otimes C) \simeq \Omega(E \otimes C')$, or, by Lemma \ref{lem:omegamonoidal}, $\Omega(E) \otimes \Omega(C) \simeq \Omega(E) \otimes \Omega(C')$.
	
	Thus the result follows if tensoring with a cofibrant dg category preserves quasi-equivalences. 
	But this is readily verified, both quasi-full faitfulness and quasi-essential surjectivity are easy to check.
\end{proof}
\begin{theorem}\label{thm:coalgmonoidal}
	The category $\ptdco$ is a monoidal model category.
\end{theorem}
\begin{proof}
	Recall the model structure on pointed curved coalgebras from \cite{Holstein6}.
The cofibrations are given by injections and a map $f: C \to C'$ is a weak equivalence exactly if $\Omega(f)$ is a quasi-equivalence.

	It remains to show that this is compatible with the closed monoidal structure from Theorem \ref{thm:ptdco_closed}. 
	The unit axiom follows from Corollary \ref{cor:tensorwe} as tensor product preserves all weak equivalences.

	It remains to check the pushout-product axiom, i.e.\ let $E \to E'$ and $C\to C'$ be cofibrations in $\ptdco$. We first check that $(E \otimes C') \amalg_{E \otimes C} (E' \otimes C) \to E' \otimes C'$ is a cofibration.
	Indeed, rewriting the pushout as a coequalizer it follows from the description of coequalizers in Lemma 3.30 of \cite{Holstein6}
	that the coequalizer of any two parallel arrows with a cone has underlying space the coequalizer in graded augmented quivers. This suffices to show that the canonical map to $E' \otimes C'$ is injective, and thus a cofibration.
	
	Next we assume that $C \to C'$ is, moreover, a weak equivalence. 
	By the first part (or by inspection) $E \otimes C \to E \otimes C'$ and $E \otimes C \to E' \otimes C$ are cofibrations.
	As $\Omega$ preserves cofibrations we see that $\Omega(E \otimes C') \amalg_{\Omega(E \otimes C)} \Omega(E' \otimes C) $ is in fact a homotopy colimit, and weakly equivalent to the homotopy colimit $(\Omega C \otimes \Omega E') \amalg^h_{\Omega C \otimes \Omega E} (\Omega C' \otimes \Omega E)$. 
	But the latter is weakly equivalent to $\Omega C \otimes \Omega E'$ as $\Omega C \otimes \Omega E \to \Omega C' \otimes \Omega E$ is a weak equivalence.
Thus the canonical map to $\Omega (C' \otimes E') \simeq \Omega C' \otimes \Omega E'$ is also a weak equivalence.
\end{proof}
\begin{rem}
	It is well-known that $\dgCat$ is not a monoidal model category. As $\Omega$ is only quasi-monoidal there is no contradiction to $\ptdco$ being a monoidal model category.
\end{rem}

In order to consider mapping spaces, we recall an adjunction between simplicial sets with the Joyal model structure (which we denote $\qCat$ for quasi-categories) and pointed curved coalgebras.
Define $F(C):= \Hom_{\ptdco}(\tilde C_*(\Delta^\bullet), C)$ where $C$ is a a pointed curved coalgebra, $K$ is a simplicial set and   $\tilde C(K)$ is the twisted chain coalgebra of $K$ (the detailed definition of is found in Section 4 of \cite{Holstein6}).
Then there is a Quillen adjunction $\tilde C_*: \qCat \leftrightarrows : \ptdcono: F$.

To consider mapping spaces we will be interested in the 'maximal subgroupoid' of an $\infty$-category.

\begin{defi}
	We call the maximal Kan subset of a quasi-category $K$ the \emph{core} of $K$.
\end{defi}
It is clear that the core is right adjoint to the inclusion of Kan complexes into quasi-categories (weak Kan complexes).

\begin{cor}\label{cor:coalgmapping}
Let $C,C'$ be pointed curved coalgebras.	Then their mapping space $\Map(C,C')$ in the model category $\ptdco$ is weakly equivalent to the core of the $\infty$-category $F\uHom(C, C')$.
\end{cor}
\begin{proof}

	We define a cosimplicial simplicial set $E^\bullet$ by letting $E^n$ be the nerve of the groupoid  with object set $[n]$ and one arrow connecting each pair of objects.
	By Section 4.1 of \cite{Dugger11}
	this is  a Reedy cosimplicial resolution of the point in the Joyal model structure i.e.\ a cosimplicial resolution of the trivial $\infty$-category.
		
	We claim that tensoring with $\tilde C_*(E^\bullet)$ gives a cosimplicial resolution in $\ptdco$.
	Indeed as $\tilde C_*$ and the tensor product are left Quillen, they preserve colimits (and thus latching objects) and cofibrations, and thus Reedy cofibrant objects.
	
	We compute 
	\begin{eqnarray*}
	\Map(C, C') &\cong &\Hom_\ptdco(C \otimes \tilde C_*(E^\bullet), C') \\
	& \cong & \Hom_\ptdco(\tilde C_*(E^\bullet), \uHom(C, C')) \\
	& \cong & \Hom_\sSet(E^\bullet, F\uHom(C,C'))
\end{eqnarray*}
by the closed monoidal structure of $\ptdco$. But the expression in the final line above is, by construction, the core of $F\uHom(C,C')$. (We may also view it as $\Map_{\qCat}(*, F\uHom(C, C'))$, which is known to be the core.)
\end{proof}

\begin{rem}\label{rem:diagram}
	We recall from \cite[Section 4.2]{Holstein6} the relation of Koszul duality and the coherent nerve $\Ncoh$ construction sending simplicial categories to quasi-categories. 
	There is the following diagram, where the inner and outer square commute up to homotopy.
	
	\[	\xymatrix@=32pt{
		\qCat\ar@<-0.5ex>[r]_{\mathfrak C}
		\ar@<-0.5ex>[d]_{\utilde C_*}
		&\sCat\ar@<-0.5ex>[l]_{\Ncoh}
		\ar@<-0.5ex>[d]_{G_*}\\
		\ptdco \ar@<-0.5ex>[r]_{\Omega}
		\ar@<-0.5ex>[u]_{F}
		&\dgCat'\ar@<-0.5ex>[l]_{\Ba} \ar@<-0.5ex>[u]_{H}
	}
	\] 
	Here the horizontal arrows are $\infty$-equivalences and downward arrows are induced by normalized chain functors, the upwards maps are right adjoints (on the level of homotopy categories).
	We also showed in \cite{Holstein6} that the dg nerve $\Ndg$, the natural functor from $\dgCat'$ to $\qCat$ which is equivalent to $\Ncoh \circ H$, factors through $\Ba$.
	
	On the left hand side of the above diagram we have simplicial sets with the Joyal model structure and pointed curved coalgebras, which both possesss monoidal model structures. In contrast, simplicial categories or dg categories on the right hand side do not have monoidal model structures. 
	
	Note that, moreover, $\tilde C_*$ has a lax monoidal structure given by the Eilenberg-Zilber map together with the isomorphism $\tilde C_*(K) \cong C_*(K)$ of curved coalgebras. 
	Thus $\tilde C_*$ also has a lax closed structure.
	This closed structure is far from being quasi-strong. Indeed, restricting to the subcategory of spaces in $\qCat$ it is clear that the map $\tilde C_*\Map(X,Y) \to \uHom(\tilde C_*X, \tilde C_*Y)$ does not become an isomorphism in the homotopy category: 
	$\uHom(\tilde C_*X, \tilde C_*Y)$ ignores the homotopy cocommutative nature of chain coalgebras and so cannot faithfully reflect maps between the corresponding spaces. 
\end{rem}

\section{Coalgebras and dg categories}\label{sec:coalgebrasandcategories}

\begin{cor}\label{cor:monoidalequivalence}
	There is an equivalence of closed monoidal model categories between $Ho(\ptdco)$ and $Ho(\dgCat)$. 
\end{cor}
\begin{proof}
	There is an equivalence of categories $Ho(\ptdco)$ and $Ho(\dgCat)$ by \cite{Holstein6}.
	This equivalence is strong monoidal by Lemma \ref{lem:omegamonoidal} and thus also closed by general principles.
\end{proof}

Note that in the statement of the corollary we could replace $\dgCat'$ by the usual $\dgCat$ as the two have equivalent homotopy categories (and the natural inclusion $\dgCat' \to \dgCat$ is quasi-strong monoidal).

In fact, before passing to homotopy categories there is a very strong relation between the monoidal model category $\ptdco$ and the model category $\dgCat$.
 
Recall that given  a model category $\mathcal M$ and a monoidal model category $\mathcal C$ we say $\mathcal M$ is a \emph{$\mathcal C$-enriched model category} \cite[Definition A.3.1.5]{Lurie11a} if $\mathcal M$ is enriched in $\mathcal C$, tensored and cotensored over $\mathcal C$ and moreover the tensor action $ \tilde \otimes : \mathcal C \times \mathcal M \to \mathcal M$ is a left Quillen bifunctor (which ensures that enrichment and cotensor are also compatible with the model structures).
This is a natural generalisation of the notion of simplicial model category.

\begin{theorem}\label{thm:cotensor}
		The category $\dgCat'$ is a $\ptdco$-enriched model category. 
\end{theorem}

\begin{proof}
	We define  the external hom of a coalgebra $C$ and dg category $D$ as $\overline{\Hom}(\Omega C, D) \coloneqq \Ba \MCp C D$ and extend to all dg categories by writing a dg category as a colimit of cobar constructions using Lemma \ref{lem:freecategories}.

	We then define the cotensoring $D^C \coloneqq \MCp C D$ and the tensoring by $C \tilde \otimes \Omega C' \coloneqq \Omega(C \otimes C')$ for dg categories in the image of $\Omega$, extended by colimits to all dg categories. 
	
	In both cases functoriality follows from the Yoneda lemma as in the proof of Theorem \ref{thm:ptdco_closed}. E.g.\ for fixed $C, D$ any map $\Omega C' \to \Omega C''$ induces a map $\Hom(\Omega C'', \MCp {C} {D}) \to \Hom(\Omega C', \MCp {C} D)$' and thus $\Hom(\Omega(C'' \otimes C), D) \to \Hom(\Omega(C' \otimes C), D)$.
		
	To show compatibility for the cotensoring we need to check that ${\oHom}(D', D^C) \cong \uHom(C, \overline{\Hom}(D', D))$. Assume first $D' = \Omega C'$.
	We have
	\begin{eqnarray*}
		\uHom(C, \oHom(\Omega C', D)) 
		&\cong& \uHom(C, \Ba \MCp {C'} D) \\
		&\cong& \uHom(C, \uHom(C', \Ba D)) \\
		&\cong& \uHom(C \otimes C', \Ba D) \\
		&\cong& \uHom(C', \uHom(C, \Ba D)) \\
		&\cong& \uHom(\Omega C', \MCp C D)
	\end{eqnarray*}
	and the last term is $\oHom(D, D^C))$ by definition.
	The general case follows by applying Lemma \ref{lem:freecategories} and observing that the construction on both sides is compatible with colimits as in the proof of Theorem \ref{thm:ptdco_closed}.

	The compatibility condition of the tensor product is $\overline{\Hom}(C \tilde \otimes D', D) \cong \uHom(C, \overline{\Hom}(D', D))$. 
	Again it suffices to show this if $D' = \Omega C'$.
	Then we have
	\begin{eqnarray*}
		\oHom(C \tilde \otimes \Omega C, D) 
		&\cong& \oHom(\Omega(C \otimes C', D)) \\
		&\cong & \uHom(C \otimes C', \Ba D) \\
		& \cong & \uHom(C, \uHom(C', \Ba D)) \\
		& \cong & \uHom(C, \oHom(\Omega C', D))
	\end{eqnarray*}
completing the proof of compatibility.
	To obtain the second line we used that  $\uHom(C, \Ba D) \cong B\MCp C D \cong \oHom(\Omega C, D)$ as pointed curved coalgebras by definition.
	
It remains to show that the enrichment is compatible with the model structures, i.e.\ that $\tilde \otimes$ is a left Quillen bifunctor, associating to a pair of cofibrations $f: C\to C'$ and $g: D \to D'$ a cofibration 
\[
f \boxempty g: (C \tilde \otimes D') \amalg_{C \tilde \otimes D} (C' \tilde \otimes D) \to C' \tilde \otimes D'
\]
that is acyclic if $f$ or $g$ is.

We show first that $f \boxempty g$ is a cofibration. By \cite[Corollary 4.25]{Hovey07} it suffices to check this only in the case that $f$ and $g$ are generating cofibrations.
But the generating cofibrations in $\dgCat'$ are images of cofibrations under $\Omega$, as was shown in the proof of \cite[Proposition 3.33]{Holstein6}.

Thus we may assume $D \to D'$ is of the form $\Omega(\tilde g)$ for a cofibration $\tilde g: E \to E'$ and we consider
\[
\Omega(C \otimes E') \amalg_{\Omega(C  \otimes E)} \Omega(C' \otimes E) \to \Omega(C' \otimes E')
\]
which is a cofibration as it is an image under $\Omega$ of $f\boxempty \tilde g$, which in turn is a cofibration by Theorem \ref{thm:coalgmonoidal}.

Similarly, to check the case one of $f$, $g$ is acyclic it suffices to check on generating trivial cofibrations in $\dgCat$. But these also lie in the image of $\Omega$, so the same argument applies.
\end{proof}

The following observation was used in the proof of Theorem \ref{thm:cotensor}, it is worth re-stating:
\begin{cor}
	The Koszul adjunction $\Omega: \ptdco \rightleftarrows \dgCat: \Ba$ is enriched in $\ptdco$, i.e.\ there is a natural isomorphism $\uHom(C, \Ba D) \cong \oHom(\Omega C, D)$ of pointed curved coalgebras enhancing the adjunction isomorphism.
\end{cor}

\begin{rem}
	In particular, note that the MC elements constructed in an ad-hoc way to prove categorical Koszul duality in \cite{Holstein6}
	are just the objects of the MC category $\MCp C D$.
\end{rem}
\begin{cor}\label{cor:dgcathom}
	The internal hom in $Ho(\dgCat)$ may be computed as $$R\uHom(D, D') \simeq \MCp {BD} {D'}.$$
\end{cor}
\begin{proof}
	On the level of homotopy categories, $L\Omega$ is a strong monoidal equivalence of categories, thus it identifies internal homs. It follows that $L\Omega(R\uHom(\Ba D, \Ba D'))$ is an internal hom in the homotopy category.  Since $\Ba D'$ is always fibrant and all pointed curved coalgebras are cofibrant we may rewrite the internal hom underived as 
	 $\Omega(\uHom(\Ba D, \Ba D'))
	 \simeq \MCp {\Ba D} {D'}$.
\end{proof}

\begin{rem}
The objects in the category $\MCp {BD} {D'}$ are unital $A_\infty$ functors $D\to D'$, either by definition, or by unravelling the relevant  definitions formulated in terms of multilinear maps cf. \cite{Canonaco19}. The morphisms in	$\MCp {BD} {D'}$ are then unital $A_\infty$ transformations between the corresponding $A_\infty$ functors. We obtain a characterization of the derived internal homs in $\dgCat$ as the category of unital $A_\infty$ functors as proposed by Kontsevich and proved rigorously in \cite{Canonaco19}.
Note that \cite{Canonaco19} considers not necessarily unital $A_\infty$ transformations between functors, but the two notions agree. This was shown in \cite[Lemma 8.2.1.3]{Lefevre03}, it also follows immediately from Lemma \ref{lem:hhbar} below, the difference between unital and non-unital transformations corresponds exactly to the difference between reduced and non-reduced Hochschild cochains.

While the embedding of dg categories into $A_\infty$ categories thus computes the correct internal hom object, it is worth recalling that the category of $A_\infty$ categories does not have a sensible monoidal structure.
\end{rem}

\begin{theorem}\label{thm:dgmapping}
	Given two dg categories $D, D'$ the mapping space $\Map(D, D')$ is weakly equivalent to the core of $\Ndg R\uHom(D, D')$.
\end{theorem}

\begin{proof}
 By the main results of \cite{Holstein6} we may compute $\Map(D, D')$ as $\Map(\Ba D, \Ba D')$ in $\ptdco$.	
	By Theorem \ref{thm:ptdco_closed} and Corollary \ref{cor:coalgmapping} this is given by the core of 
	\[F\Ba \MCp {\Ba D}{D'} \simeq \Ndg \MCp {\Ba D} {D'} \simeq \Ndg R\uHom(D,D'),\] where we used $\Ndg \simeq F \Ba$ from Theorem 4.16 in \cite{Holstein6} as well as Corollary \ref{cor:dgcathom}.
\end{proof}
\begin{rem}
The theorem may also be deduced directly from the existence of an internal hom in $Ho(\dgCat)$ and some other standard results.
We denote the left adjoint of the dg nerve by $L$ and note that $L(*) = \ground$ (considered as a dg category with one object).
Then the core of $\Ndg(R\uHom(D, D'))$ may be computed as 
\[
\Map_{\qCat}(*, \Ndg(R\uHom(D, D'))) \simeq \Map_{\dgCat}(L*, R\uHom(D, D')) \simeq \Map_{\dgCat}(\ground \otimes D, D')
\]
Here we use that the closed structure $\dgCat$ induces a weak equivalence of mapping spaces even though it is not Quillen. This follows from the Yoneda Lemma by taking hom out an arbitrary simplicial set and using the simplicial enrichment on $\dgCat$ \cite[Section 5]{Toen06}.
\end{rem}

We may compare Theorem \ref{thm:dgmapping} with To\"en's characterization of $\Map(D, D')$ as the (classical) nerve of the category of weak equivalences in the category of right quasi-representable cofibrant $D\otimes {D'}\op$-modules \cite[Theorem 1.1]{Toen06}.
In the special case of $\Map(\ground, D)$ this gives the following corollary:
\begin{cor}\label{cor:nervecore}
For a dg category $D$ the core of $\Ndg(D)$ is equivalent to the nerve of the 1-category of quasi-isomorphisms between cofibrant quasi-representable $D\op$-modules.
\end{cor}

This corollary is already interesting in the case that the dg category has one object. Then it says that the core of the simplicial Maurer-Cartan set of a dg algebra $A$ (as considered in \cite[Section 4.1]{Holstein6}) is given by the $A$-component of the nerve of the category of quasi-isomorphisms between cofibrant $A$-modules.

We have considered the Dwyer-Kan model structure on $\dgCat$ so far, but similar results apply for the Morita model structure.
We denote by $\dgCatM$ the left Bousfield localization of $\dgCat'$ at all \emph{Morita equivalences}, i.e.\ functors inducing equivalences of derived categories.

\begin{lem}
	There is a Morita model structure on $\ptdco$ whose weak equivalences are all those maps whose cobar construction is a Morita equivalence and whose cofibrations are the injections. It is Quillen equivalent to $\dgCatM$.
\end{lem}
We denote pointed curved coalgebras with the Morita model structure by $\ptdcoM$. 
\begin{proof}
	By \cite[Proposition 3.33]{Holstein6} $\ptdco$ is a left proper combinatorial model category, thus the left Bousfield localization at Morita equivalences exists.
	
	The cobar construction on $\ptdcoM$ defines a left Quillen functor to $\dgCatM$, which induces an equivalence on homotopy categories.
	Thus the Quillen equivalence $\ptdco \rightleftarrows \dgCat$ restricts to a Quillen equivalence 	
\end{proof}
\begin{lem}\label{lem:coalgmoritamonoidal}
	The tensor product makes $\ptdcoM$ into a monoidal model category whose homotopy category is monoidally equivalent to $Ho(\dgCatM)$. 
	\end{lem}
\begin{proof}
	The closed monoidal structure is the one considered in Theorem \ref{thm:coalgmonoidal}.
	
		We first note that Lemma \ref{lem:omegamonoidal} implies that $\Omega$ is quasi-strong monoidal also for the Morita model structures. 
		We then observe that tensor product of dg categories preserves Morita equivalences. This is well-known, see Exercise 32 in \cite{Toen07a}.
\end{proof}

\begin{prop}
	The category $\dgCatM$ is $\ptdcoM$-enriched model category.
\end{prop}
\begin{proof}
	Given $f: C \to C'$ in $\ptdcoM$ and $g: D \to D'$ in $\dgCatM$ we consider
	\[
	f \boxempty g: (C \tilde \otimes D') \amalg_{C \tilde \otimes D} (C' \tilde \otimes D) \to C' \tilde \otimes D'
	\]
		As the cofibrations in $\dgCat$ and $\dgCatM$ agree we see that $f \boxempty g$ is a cofibration if $f$ and $g$ are, by the proof of Theorem \ref{thm:cotensor}.
		
	We now have to check that $f \boxempty g$
	is a Morita equivalence if $f$ or a $g$ is a (Morita) acyclic cofibration.	
	It would suffice to show that all generating acyclic cofibrations in $\dgCatM$ lie in the image of $\Omega$. We expect this to be true, but to avoid excessive computations we may use another argument.
	
	By imitating the proof of Theorem \ref{thm:coalgmonoidal}, it suffices to show that the action $\tilde \otimes$ preserves Morita equivalences in $\ptdcoM$ and $\dgCatM$. 
	Morita equivalences will be denoted by $\simeq$ in the following.
	
	The proof uses the simple observation that for any dg category $D$ we have a natural quasi-equivalence $\Omega \Ba D \xrightarrow{\sim} D$, thus by Theorem \ref{thm:cotensor} we have $C \tilde \otimes \Omega \Ba D \simeq C \tilde \otimes D$ for any $C \in \ptdco$. 
	
	Let $C \xrightarrow{\sim} C'$ be an acyclic cofibration in $\ptdcoM$.
	As $\Omega$ is quasi-strong monoidal and $\otimes$ preserves Morita equivalences, see Lemma \ref{lem:coalgmoritamonoidal}, we have the following zig-zag of Morita equivalences:
	 \[
	C \tilde \otimes D 
	\xleftarrow{\sim} C \tilde \otimes \Omega \Ba D  
	\xleftarrow{\sim} \Omega C \otimes \Omega \Ba D
	\xrightarrow{\sim} \Omega C' \otimes \Omega \Ba D
	\xrightarrow{\sim} C' \otimes \Omega \Ba D
	\xrightarrow{\sim} C' \tilde \otimes D'
	\]
	and thus $C \tilde \otimes D \simeq C' \tilde \otimes D$.
	
Similarly, let $D \to D'$ be a Morita equivalence. 
As $\Omega \Ba D \simeq D$ it follows from 2-out-of-3 that $\Omega \Ba D \simeq \Omega \Ba D'$.
	Putting this together with the previous observations we get the following Morita equivalences
	 \[
	 C \tilde \otimes D 
	 \xleftarrow{\sim} C \tilde \otimes \Omega \Ba D 
	\xleftarrow{\sim} \Omega C \otimes \Omega \Ba D
	 \xrightarrow{\sim} \Omega C \otimes \Omega \Ba D'
	 \xrightarrow{\sim} C \tilde \otimes \Omega \Ba D' 
	 \xrightarrow{\sim} C \tilde \otimes D'.
	 \]
	 Thus  $C \tilde \otimes D \simeq C \tilde \otimes D'$ and the proposition follows. 
	\end{proof}

\begin{cor}
	The internal hom in $Ho(\dgCatM)$ is computed by 
	\[R\uHom(D, D') \simeq \MCp {BD} {D'}\] whenever $D'$ is Morita fibrant.
\end{cor}
\begin{proof}
	As for Corollary \ref{cor:dgcathom}, except that in order to ensure $\Ba D'$ is fibrant we need to assume $D'$ is fibrant in $\dgCatM$.
\end{proof}

\section{Hochschild cohomology}\label{sec:Hochschild}
One key application of the construction of the derived internal hom of dg categories is the computation of homotopy groups in terms of Hochschild cohomology by To\"en \cite{Toen06}.
We recreate this computation in our setting in somewhat greater generality.

\begin{defi}\label{defi:hhd}
	The Hochschild cochain complex of a dg category $D$  with coefficients in a bimodule $M: D \otimes D\op \to \Ch$ is defined as 
	\[\HHH(D, M)  \coloneqq R\Hom_{D\otimes D\op}(D, M).\]
\end{defi}
It is well-known that	$\HHH(D, M)$ may be computed by the Hochschild complex of the dg category $D$ with coefficients in $M$. 
Writing $D(d_0,d_1)$ for $\Hom_D(d_0, d_1)$ for better legibility this complex is 
	\[\prod M(d_0, d_1)\to \prod \Hom(D(d_0, d_1), M(d_0, d_1))\to \prod\Hom(D(d_1, d_2) \otimes D(d_0, d_1), M(d_0, d_2)) \cdots \to     ,
\]
with a differential induced  by the internal differentials, composition in $D$ and the action of $D$ on $M$. 
One may equivalently compute with the reduced Hochschild complex replacing $D$ by $\overline D$ everywhere.

We specialise now to the case where $M$ is given by another dg category $D'$ with a pair of functors $F, G: D \to D'$, i.e.\ we consider the bimodule ${}_FD'_G$, or by abuse of notation just $D'$,  sending $d_1 \otimes d_0$ to $\Hom_{D'}(F(d_0), G(d_1))$.

Any dg functor $F: D \to D'$ gives rise to an object in $R\uHom(D, D')$ in a natural way. We use this to state the following lemma:
\begin{lem}\label{lem:hhbar}
	Let $F, G: D \to D'$ be functors of dg categories, then Hochschild cohomology of $D$ with coefficients in $D'$ is given by 
	\[\HHH(D, {}_{F}D'_G)  \cong \Hom_{R\uHom(D,D')}(F,G).\]
\end{lem}
\begin{proof}
	By definition 
	$\uHom_{R\Hom(D,D')}(F,G)\cong \Hom_{\MCp {\Ba D} {D'}}(F, G)$ is the twisted hom space $\Hom_{\uHom(\Ba D, D')}(F, G)^{[\xi_F, \xi_G]}$ where $\xi_F, \xi_G$ are the MC elements in the convolution category corresponding to $F, G: D \to D'$.
	Unravelling definitions, we recognize this as the reduced Hochschild complex which we may write as
	\[D'\to\Hom_{D'_0}(\overline{D}, D')\to  \Hom_{D'_0}(\overline{D} \boxempty_{D_0} \overline{D}, D')\to\cdots  
	\]
	with differential induced by internal differentials, the composition in $D$ and action of $D$ on $D'$ (the latter corresponding to the twist by $\xi_F$ and $\xi_G$).
\end{proof}
If $F = G = \id_D$ this specializes to the well-known equivalence $\HHH(D) \cong \End_{R\uHom(D,D)}(\id)$.

\begin{rem} 
	One can also directly relate $\Hom_{R\uHom(D, D')}(F,G)$ to $R\Hom_{D \otimes D\op}(D, {}_FD'_G)$ without reference to any resolutions, as follows. 

	We use that the functor $F: D \to D'$ gives rise to a $D \otimes {D'}\op$-bimodule $D'_F$ sending $d, d' \mapsto \Hom_{D'}(d', F(d))$.
	The functor $\hat F = \id \otimes F \op: D \otimes D\op \to D \otimes {D'}\op$ induces an adjunction on module categories.
	Considering the $D\otimes D\op$-module $D$ and the $D\otimes {D'}\op$-module $D'_{G}: d \otimes d' \mapsto \Hom(d', d)$ we have
	\[
	R\Hom_{D\otimes D\op}(D, \hat F^*(D'_G) \cong R\Hom_{D \otimes {D'}\op}(\hat F_! D, D'_{G})
	\]
	By definition $\hat F^*({}_GD')$ is ${}_FD'_G$ and moreover one may check $\hat F_!D = D'_{F}$, for example by rewriting $\hat F_!$, which by definition is a left Kan extension, as a coend and computing it.  

The proof is completed by recalling from \cite{Toen06} that the functor category $R\uHom(D, D')$ is weakly equivalent to a full subcategory of (fibrant cofibrant) $D \otimes {D'}\op$-modules, and in particular the map $F \mapsto D'_F$ induces weak equivalences of hom spaces.
Thus we have 
\[
\Hom_{R\uHom(D, D')}(F, G) \simeq R\Hom_{D \otimes {D'}\op}(D'_{F}, D'_G) 
\simeq R\Hom_{D, D}(D, {}_FD'_G).
\]
This is nothing but a categorical version of the classical formula $\HH^*(A, B) \cong \operatorname{Ext}^*_{A \otimes B\op}(B, B)$ for algebras $A$ and $B$.
\end{rem}

We now specialise to $F=G$ and consider the bimodule $D' = {}_G D' _G$.

\begin{theorem}\label{thm:hhcomputation}
	Let $G: D \to D'$ be a functor of dg categories. 
	We then have $\HH^0(D,D')^\times \cong \pi_1(\Map(D,D'), G)$ and 
	$\HH^{i}(D, D') \cong \pi_{1-i} (\Map(D,D'), G)$ for $i < 0$.
\end{theorem}

\begin{proof}
By Theorem \ref{thm:dgmapping} we have $\Map(D, D')$ weakly equivalent to the core of $\Ndg R\uHom(D, D')$.
	
	It now follows, cf.\ \cite[Remark 1.3.1.12]{Lurie11},
	that the Dold-Kan image of $\Hom_{R \uHom(D,D')}(G,G)$ is the infinity categorical mapping space from $G$ to $G$ in $F\Ba \MCp {\Ba D} {D'}$.
Taking homotopy groups and using Lemma \ref{lem:hhbar} we have $\HH^{i}(D, D') = \pi_{-i}(\Omega\Map(D,D'), G)$ for $i > 1$ as the higher homotopy groups of mapping spaces are unaffected by taking the core.
Finally $\HH^0(D,G)$ is given by $\pi_0\Map_{F\Omega \MCp {BD} D'}(G,G)$ and taking units on both sides we obtain $\HH^0(D, D')^* \cong \pi_0\Map_{\Map(D,D')}(G,G) = \pi_1(\Map(D,D'), G)$.
\end{proof}

\begin{cor}
	[{\cite[Corollary 8.2]{Toen06}}]
	$\HH^i(D) \cong \pi_{1-i}(\Map(D, D), \id)$ for $i < 0$ and $\HH^0(D)^\times \cong \pi_1(\Map(D,D), \id_D)$.
\end{cor} 

\bibliography{./biblibrary2}

\begin{thebibliography}{10}

\bibitem{Anel13}
{\sc M.~{Anel} and A.~{Joyal}}, {\em {Sweedler Theory for (co)algebras and the
  bar-cobar constructions}}, arXiv e-prints,  (2013),
  \href{http://arxiv.org/abs/1309.6952}{{ arXiv:1309.6952}}.

\bibitem{Barr00}
{\sc M.~Barr and C.~Wells}, {\em Toposes, triples, and theories},
  Springer-Verlag, 2000.

\bibitem{Canonaco19}
{\sc A.~Canonaco, M.~Ornaghi, and P.~Stellari}, {\em {Localizations of the
  category of $ A_\infty $ categories and internal Homs}}, Doc. Math., 24
  (2019), pp.~2463--2492.

\bibitem{Stel12}
{\sc A.~Canonaco and P.~Stellari}, {\em Fourier-{M}ukai functors: a survey}, in
  Derived categories in algebraic geometry, EMS Ser. Congr. Rep., Eur. Math.
  Soc., Z\"{u}rich, 2012, pp.~27--60.

\bibitem{Stel15}
\leavevmode\vrule height 2pt depth -1.6pt width 23pt, {\em Internal {H}oms via
  extensions of dg functors}, Adv. Math., 277 (2015), pp.~100--123.

\bibitem{CE56}
{\sc H.~Cartan and S.~Eilenberg}, {\em Homological algebra}, Princeton
  University Press, Princeton, N. J., 1956.

\bibitem{Dugger11}
{\sc D.~Dugger and D.~I. Spivak}, {\em {Mapping spaces in quasi-categories}},
  {Algebraic \& Geometric Topology}, 11 (2011), pp.~263 -- 325.

\bibitem{Eilenberg66}
{\sc S.~Eilenberg and J.~C. Moore}, {\em {Homology and fibrations. I.
  Coalgebras, cotensor product and its derived functors}}, Comment. Math. Helv,
  40 (1966), pp.~199--236.

\bibitem{bjor}
{\sc B.~{Eurenius}}, {\em {Enriched Koszul duality}}, PhD thesis, Lancaster
  University, in preparation.

\bibitem{Holstein6}
{\sc J.~Holstein and A.~Lazarev}, {\em Categorical {K}oszul duality}, Adv.
  Math., 409 (2022), pp.~Paper No. 108644, 52.

\bibitem{Hovey07}
{\sc M.~Hovey}, {\em {Model categories}}, no.~63 in Mathematical Surveys and
  Monographs, American Mathematical Society, 2007.

\bibitem{Keller05}
{\sc B.~Keller}, {\em {$A$}-infinity algebras, modules and functor categories},
  in Trends in representation theory of algebras and related topics, vol.~406
  of Contemp. Math., Amer. Math. Soc., Providence, RI, 2006, pp.~67--93.

\bibitem{Lefevre03}
{\sc K.~{Lef{\`e}vre-Hasegawa}}, {\em {Sur les A-infini cat{\'e}gories}}, PhD
  thesis, Universit\'{e} Paris 7 - Denis Diderot, 2003,
  \href{http://arxiv.org/abs/math/0310337}{{ arXiv:math/0310337}}.

\bibitem{Lurie11a}
{\sc J.~Lurie}, {\em {Higher Topos Theory}}, vol.~170 of Annals of Mathematics
  Studies, Princeton University Press, 2011.

\bibitem{Lurie11}
\leavevmode\vrule height 2pt depth -1.6pt width 23pt, {\em {Higher Algebra}}.
\newblock Available from the author's website as
  www.math.harvard.edu/$\sim$lurie/papers/higheralgebra.pdf, 2017.

\bibitem{Positselski11}
{\sc L.~Positselski}, {\em Two kinds of derived categories, {K}oszul duality,
  and comodule-contramodule correspondence}, Mem. Amer. Math. Soc., 212 (2011),
  pp.~vi+133.

\bibitem{Rivera19}
{\sc M.~{Rivera} and M.~Zeinalian}, {\em {The colimit of an infinity local
  system as a twisted tensor product}}, Higher Structures, 3 (2019), pp.~1--24,
  \href{http://arxiv.org/abs/arXiv:1805.01264}{{ arXiv:1805.01264}}.

\bibitem{Tabu05}
{\sc G.~Tabuada}, {\em Une structure de cat\'{e}gorie de mod\`eles de {Q}uillen
  sur la cat\'{e}gorie des dg-cat\'{e}gories}, C. R. Math. Acad. Sci. Paris,
  340 (2005), pp.~15--19.

\bibitem{Tabuada10a}
{\sc G.~Tabuada}, {\em {Homotopy theory of dg-categories via localizing pairs
  and Drinfeld's dg quotient}}, Homology, Homotopy Appl., 12 (2010),
  pp.~187--219.

\bibitem{Toen06}
{\sc B.~To\"{e}n}, {\em {The homotopy theory of dg-categories and derived
  Morita theory}}, Invent. Math., 167 (2006), pp.~615--667.

\bibitem{Toen07a}
\leavevmode\vrule height 2pt depth -1.6pt width 23pt, {\em Lectures on
  dg-categories}, in Topics in algebraic and topological {$K$}-theory,
  vol.~2008 of Lecture Notes in Math., Springer, Berlin, 2011, pp.~243--302.

\bibitem{Wolff74}
{\sc H.~Wolff}, {\em {V-cat and V-graph}}, Journal of Pure and Applied Algebra,
  4 (1974), pp.~123--135.

\end{thebibliography}

\end{document}